\documentclass[11pt,a4paper]{amsart}
\usepackage{geometry}
\usepackage{amsmath}
\usepackage{amssymb}
\usepackage{amsthm}
\usepackage{xcolor}

\newtheorem{theorem}{Theorem}
\newtheorem{lemma}[theorem]{Lemma}

\usepackage{amsmath,amsthm,amssymb,amsfonts}
\usepackage{epsfig}
\usepackage{graphicx}
\usepackage{bm}
\usepackage{dsfont}
\numberwithin{equation}{section}

\def\blem{\begin{lemma}}
\def\elem{\end{lemma}}
\def\bthm{\begin{theorem}}
\def\ethm{\end{theorem}}

\def\ve{\varepsilon}
\def\R{\mathbb{R}}
\def\N{\mathbb{N}}

\def\bproof{\begin{proof}}
\def\eproof{\end{proof}}

\def\bald{\begin{aligned}}
\def\eald{\end{aligned}}

\def\Hu{H_1}
\def\Hd{H_2}
\def\Hun{H_{1,n}}
\def\Hdn{H_{2,n}}
\def\Hin{H_{i,n}}
\def\Hi{H_{i}}

\def\ds{\rightarrow}

\def\A{\mathbb{A}}

\def\T{\mathcal{T}}

\def\alp{\alpha}
\def\al0{\alpha_0}

\def\alu{\alpha_1}
\def\ald{\alpha_2}
\def\al0{\alpha_0}

\def\ali{\alpha_i}

\def\Iu{I_1}
\def\Id{I_2}
\def\I0{I_0}

\def\Ii{I_i}

\def\l0{l_0}

\def\li{l_i}

\def\b0{b_0}
\def\bi{f_i}
\def\bu{f_1}
\def\bd{f_2}

\def\ga{\gamma}

\def\ga0{\gamma_0}

\def\apg{\left\{}
\def\chg{\right\}}

\newcommand{\be}{\begin{equation}}
\newcommand{\ee}{\end{equation}}
\newcommand{\baa}{\begin{array}}
\newcommand{\eaa}{\end{array}}
\newcommand{\ba}{\begin{eqnarray}}
\newcommand{\ea}{\end{eqnarray}}

\newcommand{\ban}{\begin{eqnarray*}}
\newcommand{\ean}{\end{eqnarray*}}

\newcommand{\dis}{\displaystyle}
\newtheorem{theo}{\bf Theorem}[section]
\newtheorem{lem}[theo]{\bf Lemma}
\newtheorem{pro}[theo]{\bf Proposition}

\newtheorem{defi}[theo]{\bf Definition}
\newtheorem{rem}[theo]{\bf Remark}

\title[(HJ)-equation with $t$-measurable Hamiltonians on 1-dimentional junction]{On Hamilton Jacobi equations with time measurable Hamiltonians posed on a 1-dimensional junction}

\author[A. Briani]{Ariela Briani}
\address[A. Briani]{Université de Tours, 
Université d’Orléans, CNRS, IDP, UMR 7013, Tours, France.}
\email{ariela.briani@univ-tours.fr}

\thanks{A. Briani was partially supported by l’Agence Nationale de la Recherche (ANR), project
ANR-22-CE40-0010 COSS}

\begin{document}
\begin{abstract}

In this paper, we study evolutive Hamilton–Jacobi equations with Hamiltonians that are discontinuous in time, posed on a simple network consisting of two edges on the real line connected at a single junction. \\
We introduce a notion of (flux-limited) viscosity solution for Hamiltonians $H_i=H_i(t,x,p), (i=1,2$) that are assumed to be only measurable in $t$, with $H_i(\cdot,0,0)$ in $ L^1(0,T)$. 
The flux limiter,  $A=A(t)$, acting at the junction, is not required to be continuous but only in $L^\infty(0,T)$. \\
In the case of convex Hamiltonians, we prove a comparison principle and establish an existence result via the construction of an  optimal control problem. Generalizations to the nonconvex case and to more general networks are also discussed.

\end{abstract}

\maketitle 

\tableofcontents
\allowdisplaybreaks

\section{Introduction} 

In this paper, we study viscosity solutions of Hamilton–Jacobi (HJ) equations featuring two distinct types of discontinuities. On the one hand, the Hamiltonians are assumed to be merely measurable with respect to the time variable $t$.  On the other hand, they exhibit spatial discontinuities arising from the underlying structure of the problem. Precisely, we consider a simple network  with two different Hamiltonians $H_i$ ($i=1,2$) on the line with one junction. Unless otherwise specified, all solutions are understood in the viscosity sense.

Starting from the pioneering work of H. Ishii in 1985 \cite{Is85}, it was clear that  when dealing with the {\it discontinuous case}- meaning a situation where both the Hamiltonian and the solution are allowed to be discontinuous-  for classical of HJ-equations of the form 
\be \label{EQ0}
 u_t + H(t,x,u,Du)=0 \mbox{ for } (x,t) \in \R^N \times (0,T)
\ee
an {\it ad-hoc} treatment has to be given to discontinuities on  the $t$-variable. This is mainly due to the special position of the $t$-derivative in the equation.  In fact, H. Ishii himself studied this case separately from other situations involving discontinuities. Incidentally, let us notice that it is natural to assume that solutions to \eqref{EQ0} are continuous (in both variables), for, roughly speaking, one has to integrate $H$ in time. In the sequel we will refer to the {\it t-measurable case} when the Hamiltonian $H$ is  measurable in $t$ and continuous in the remaining variables. \\
We think it is very interesting to recall here the key idea behind the first definition as it is explained  by H. Ishii himself (see \cite[Definition 7.1, p.52]{Is85}) in the simplest case of  ODEs:  \\  
{\em 
Consider  $u^\prime(t)+H(t)=0$ for $t \in (0,T)$ and  let $H(t)=g(t)+b(t)$ with $g \in C(0,T)$ and $b \in L^1(0,T)$.  \\
 The key remark is that  an absolutely continuous function $u$ solves $u^\prime(t)+H(t)=0$  a.e. in $(0,T)$ if and only if $u \in C(0,T)$ and $v(t)=u(t)+\int_0^t b(s) \: ds $ is a viscosity solution of $v^\prime (t)+g(t)=0$ in $(0,T)$.  }
 
In 1987, P.L. Lions and B. Perthame  provided three equivalent formulations of the notion of solution given by H. Ishii (see \cite{LiPer87}).  
Extension to the second order case have been studied by D. Nunziante in \cite{Nu90}, \cite{Nu92}.  With these definitions in hand, it has been common belief in the community that, roughly speaking, when a (comparison, existence, regularity, . . . ) result holds true when the Hamiltonian is continuous in  \eqref{EQ0} it is still valid when one weakens the hypothesis of continuity in $t$ to mere measurability. 
However, to prove each  parallel improvement for t-measurable Hamiltonians may appear as a bit tedious  job because of the necessarily involved technicalities. 

In \cite{BrRam} the author and F. Rampazzo studied approximation  general results that allows to establishing comparison, existence, and regularity results by deriving them from the analogous results for the case of continuous Hamiltonians for equation \eqref{EQ0}. In this paper we exploit these ideas to study HJ-equations on regional problem or networks in the {\it t-measurable case}. 

When the Hamiltonian is $t$-continuous, starting from 2010's a lot of work has been devoted to the study of deterministic control problems and HJ-equations on networks or involving discontinuities in space and, more precisely problems where dynamics and running costs may be completely different in different parts of the domain. We chose to recall here only the articles that are more related to our work, a complete bibliography can be found in the book \cite{BaChas} where all these approaches are collected, links between them are explained and generalizations are provided.
\\  
The author, with G. Barles and E. Chasseigne studied the so called regional control problems in \cite{BBC1}, \cite{BBC2} while C. Imbert and R. Monneau,   with H. Zidani, introduced the concept of Flux-limited solutions for networks in \cite{ImMo0},   \cite{ImMo1}, \cite{ImMo2}. We will refer here mainly to \cite{BBCI} where the link between the two approaches is studied to define Flux-limited solutions on regional problems. \\ 
At the same time Y. Achdou, F. Camilli, A. Cutri and N. Tchou studied HJ-equations on networks in \cite{ACCT} and P.L. Lions ans P.E. Souganidis defined solution for junctions problems with Kirchoff-type conditions in \cite{LISou1}, \cite{LISou2}. Results on stratified problems have been firstly introduced by A. Bressan and Y. Hong in \cite{BreHong} and  then generalized by G. Barles and E. Chasseigne in \cite{BaChas2}, \cite{BaChas}.

Our first motivation is the optimal control problem studied by P. Cardaliaguet et P. Souganidis in \cite{CaSoug}, where merely measurable flux limiters are considered to model traffic flow through a junction on the line. For their goals, properties of map $u^A$  in \cite[Definition 2.1]{CaSoug} obtained via a representation formula, suffices. However, it is natural to  aim to give a definition of solution for the HJ-equation related to this formula.  Precisely, 
\be \label{eqHJmodIntro}
\left\{
\begin{array}{lcl}
u_t +\Hu(u_x )=0 & \mbox{in} & (0,+\infty) \times (0,T)   \\
u_t +\Hd(u_x)=0 & \mbox{in}  & (-\infty,0) \times (0,T)  \\
u_t(0,t)+ \max\{  A(t), \Hu^- (u_x(0^+,t) ), \Hd^+(u_x(0^-,t))\}=0& \mbox{ in }  &  (0,T)  \\
u(x,0)= u_0(x)& \mbox{ in } &  \R. 
\end{array}
\right.
\ee
where  the Hamiltonians are regular and uniformly convex and  $H_i^+$, $H_i^-$  denote the increasing  and decreasing part of $H_i$, respectively. We will refer to this case as the model problem.

Thus, the first problem we addressed is how to generalize the definition of viscosity solution for  \eqref{eqHJmodIntro}  assuming only $A(\cdot) \in L^\infty(0,T)$.  When the flux limiter $A$ is continuous,  starting from \cite{ImMo0} -\cite{ImMo1}  the viscosity solutions' theory has been developped mainly in terms of {\it Flux limited solutions}.  Here, we generalize this definition following  Ishii in \cite{Is85}. Roughly speaking,  the key ingredients are:  \\
\indent{--Test function will be the sum of two functions: one of the form $ \int_0^t b(s) ds$ with $b \in L^1(0,T)$ and the other in the set  $PC^1(\R \times (0,T))$ considered in  \cite{ImMo1} (see Definition \ref{defHcont2} below).} \\
\indent{-- Inequalities at touching point will not be asked to be fulfilled by the Hamiltonians but  by any  time continuous function $G$  that is {\it not $L^\infty$-far in the right direction} from the Hamiltonians  plus the test  function $b$.}

 Since in  \eqref{eqHJmodIntro}  the Hamiltonian $H_i$ only depends on the gradient variable  we need to adapt the definition  only at the  junction, i.e. when $x=0$.
Thus, the subsolution (resp. supersolution) condition for a continuous bounded  function $u$ will read as follows:   let   $b \in L^1(0,T)$ and $\psi \in PC^1(\R \times (0,T))$  and let 
$G\in C([0,T] \times \R \times \R)$ satisfy for some $\delta >0$, 
$$
 -b(t)+G (t,p_1,p_2) \leq \max\{ A(t), \Hu^+ (p_1), \Hd^-(p_2) \}       \quad    (resp. \geq )  
$$
 for $|p_1 - \psi_x(0^-,t_0) | \leq \delta$,  $|p_2 - \psi_x(0^+,t_0) | \leq  \delta$,  and a.e.  $|t - t_0| \leq \delta$.  \\
 If $(0,t_0)$ is  local maximum (resp. minimum) point of  $\dis u(x,t) - \big( \int_0^t b(s) ds + \psi(x,t)\big)$ then 
$$
 \begin{array}{c}
  \psi_t(0,t_0)+ G(t_0, \psi_x(0^-,t_0),\psi_x(0^+,t_0)) ) \leq 0  \quad    (resp. \geq  0) \: . \\
 \end{array}
$$

In this paper we consider the case of equation \eqref{eqHJmodIntro} with Hamiltonians depending also on space and time, i.e. $H_i=H_i(t,x,p)$, $i=1,2$.
Precisely, assuming only $t$-measurability for the Hamiltonians $H_i$ and  the flux limiter  $A$, we study the problem
\be
\label{eqHJintrp2}
\left\{
\begin{array}{lcl}
u_t+\Hu(t,x,u_x)=0 & \mbox{in} & (0,\infty) \times (0,T)   \\
u_t+\Hd(t,x,u_x)=0 & \mbox{in}  &  (-\infty,0) \times (0,T)  \\
u_t(x,0)+ \max\{  A(t), \Hu^- (t,0,u_x(0^+,t) ), \Hd^+(t,0,u_x(0^-,t))\}=0& \mbox{ in }  &  (0,T)  \\
u(x,0)= u_0(x)& \mbox{ in } &  \R. 
\end{array}
\right.
\ee

In this case we not only have to adapt the definition of Flux limited solution outside the junction but also to modify the above condition on the junction. This is the aim of Section   \ref{onedimconvex} where our definition of  {\em Flux Limited t-Measurable (FL-tm)} solution is introduced and  discussed  (see Definition \ref{defHmismod2} and Remarks \ref{confaltre},  \ref{disdipinx}).  \\
In line with our first motivation of giving a definition of the equation \eqref{eqHJmodIntro} studied in \cite{CaSoug},  we assume the Hamiltonians to be convex on the gradient variable. This hypothesis will allow us to exploit the construction of an optimal control problem to obtain an existence result.   Indeed, in Theorem \ref{teoexismod2} we will prove that the value function of an optimal control problem is a  (FL-tm)-solution of  \eqref{eqHJintrp2} as in Definition  \ref{defHmismod2}.

With the aim of deriving uniqueness results for equation \eqref{eqHJintrp2}, in Theorem \ref{compmod2LOC} we prove a comparison result for (FL-tm) sub and super solution. As explained in Section \ref{comp2} the proof is based an approximating technique with the same ideas as  in \cite{BrRam} and therefore two results are essentially needed. On one hand  it has to be possible to  construct  {\it good approximating} continuous Hamiltonians and flux limiter, i.e., roughly speaking a $t$-continuous approximating sequence that are somehow uniform w.r.t.  all other variables,  and, on the other,   a comparison result has to hold for the corresponding structure in the $t$-continuous case.

The advantage of this kind of proof is that this can be easily adapted to more general Hamiltonians and/or  geometries.  Indeed,  in Section \ref{gene}  we discuss how to 
use these ideas to deal with the following cases: we  weaken the convexity assumption on the $H_i$, we add a dependence on  the $u$-variable, i.e. $H_i=H_i(t,x,u,p)$,  and we 
discuss how to generalize the definition to problems with multiple junctions, i.e. on networks. \\
Let us point out here that when dealing with the model problem, equation  \eqref{eqHJmodIntro},  the comparison result can be proved assuming $H_i(\cdot)$  to be only continuous and coercive. This is possible thanks to the corresponding result for continuous flux limiter proved by  N. Forcadel and R. Monneau in  \cite[Theorem 6.1]{ForMon}.

With these  ideas,  we believe that one can study the $t$-measurable dependent case for more general geometries or  junction conditions, as for example to problem problems in $\R^N$ with discontinuity along an hyperplane or 
problems or Kirchhoff type conditions (see \cite{BaChas} for a comprehensive description of the results in the $t$-continuous case). 
The discussion on how to adapt the definition in  these cases is beyond the scope of this work.

This paper is organized as follows: In Section  \ref{onedimconvex} we set our assumptions and give the  definition of Flux-limited solution in the $t$-measurable case for equation  \eqref{eqHJintrp2}. In Section \ref{comp2} we prove a comparison result in this framework (Theorem \ref{compmod2LOC}) and discuss possible variants in  Remark \ref{gendim1}.
Section \ref{submodoc2} is devoted  to the construction of an optimal control problem whose value function is proved to be a (FL-tm) solution of \eqref{eqHJintrp2}. Finally, in Section \ref{gene}  we discuss how to weaken the convexity assumption in the  comparison result  and how to generalize these ideas to deal with networks.


\section{The definition of Flux limited $t$-measurable solution} \label{onedimconvex}

In this section we set the problem and introduce our  definition of  Flux Limited viscosity solution in the case of $t$- measurable Hamiltonian that we call  {\bf Flux Limited t-measurable (FL-tm)}.

Let $\Omega_1:=(0,\infty)$, $\Omega_2:=(-\infty,0)$ and $T>0$.  

We will call a $t$-measurable modulus a non-negative function $\omega=\omega(t,k)$ continuous, subadditive, non-decreasing in $k$ and measurable in $t$ such that $\omega(\cdot,k) \in L^1(0,T)$ for all $k$ and $\omega(t,0)=0$ a.e. 
If moreover,  $ \lim_{\delta \ds 0}  \int_0^T \omega(t,\delta) dt =0$ we will say that $\omega$ is a $t$-integrable modulus.

We assume:  
\be \label{ipdatoiniziale2} 
\mbox{The initial condition } u_0 : \R \ds \R  \quad \mbox{is Lipschitz continuous.} \\
\ee
The Hamiltonians $H_i : [0,T] \times \R  \times \R \ds \R$, $i=1,2$ 
\be \label{ipHami2reg}
H_i(t,x,p)  \mbox{ are measurable in  } t   \mbox{ and continuous in } (x, p). 
\ee
\be \label{ipHami2int}
H_i(t,0,0)  \in  L^1(0,T).
\ee
For each $R >0$, there exists $C_1=C_1(R) >0$ and a $t$-measurable modulus $\omega_R$  such that  
\be  \label{ipHam2inX} 
| H_i(t,x,p) - H_i(t,y,p)| \leq C_1 |x-y|  |p|+ \omega_R(t, |x-y|)  \quad  \forall \: |x|,|y| \leq  R \, , p \in \R \, , t \in [0,T] \,. 
 \ee
There exists  $C_2 >0$ such that 
\be  \label{ipHam2inP} 
| H_i(t,x,p) - H_i(t,x,q)| \leq C_2 |p-q| \quad \forall p,q \in \R, x \in \R, t \in [0,T].
 \ee
Moreover, 
\be \label{ipHami2Conv}
H_i(t,x, \cdot) : \R \ds \R \mbox{ are  convex  and coercive uniformly with respect to } (t,x).
\ee

Following  \cite{ImMo1} (see also \cite{BBCI})
we write $H_i^+$ and $H_i^-$ for the nondecreasing and nonincreasing part of $H_i(t,x,\cdot)$ respectively, that is, 
$$
H_i^+(t,x,p): = \left\{
\begin{array}{ll}
\dis \min_p H_i (t,x,p)  & \mbox{ if } p \leq \hat{p}^ i(t,x) \\
H_i (t,x,p) & \mbox{ otherwise }
\end{array}
\right.
H_i^- (t,x,p): = \left\{
\begin{array}{ll}
H_i (t,x,p)  & \mbox{ if } p \leq \hat{p}^ i (t,x) \\
  \dis \min_p H_i (t,x,p) & \mbox{ otherwise }
\end{array}
\right.
$$ 
where $\hat{p}^ i(t,x) $ is a  point such that $H_i (t,x,\hat{p}^ i(t,x) )= \min_p H_i (t,x,p)$.

We set, for each $t \in [0,T]$,
\be \label{ipA02} 
   \dis A_0(t):= \max\Big\{\min_{p \in \R} \Hu(t,0,p), \min_{p  \in \R} \Hd(t,0,p)\Big\} 
   \ee
and we consider the flux limiter $A : [0,T] \ds \R$ such that 
\be \label{ipA2} 
 A\in L^\infty([0,T]) \mbox{ and  } A(t) \geq A_0(t) \mbox{ for a.e. } t \in [0,T]. 
\ee

Under this set of assumptions we will define and study {\bf (FL-tm)} solution of 
\be \label{eqHJmod2}
\left\{
\begin{array}{lcl}
u_t(x,t)+\Hu(t,x,u_x(x,t))=0 & \mbox{in} & \Omega_1 \times (0,T)   \\
u_t(x,t)+\Hd(t,x,u_x(x,t))=0 & \mbox{in}  & \Omega_2 \times (0,T)  \\
u_t(x,0)+ \max\{  A(t), \Hu^- (t,0,u_x(0^+,t) ), \Hd^+(t,0,u_x(0^-,t))\}=0& \mbox{ in }  &  (0,T)  \\
u(x,0)= u_0(x)& \mbox{ in } &  \R. 
\end{array}
\right.
\ee

In the case of  $t$-continuous Hamiltonians $H_i$ and flux limiter $A$ one can  define solution following Imbert et Monneau  \cite{ImMo1}-\cite{ImMo2}. In particular, they proved in  \cite[Theorem 5.9]{ImMo1} existence and uniqueness for  {\it Flux limited}  solution of  equation \eqref{eqHJmod2}. The latest has also been studied in \cite{BBCI} where a different proof of comparison result is given. We will also refer to the results collected in \cite[Chapter 14]{BaChas}. \\
For the sake of completeness we recall below this definition where the following set of test functions is considered 
: $PC^1(\R \times (0,T))$ is the space of piecewise fonctions $\psi \in C(\R \times (0,T))$ such that there exist $\psi_1 \in C^1(\overline{\Omega_1} \times (0,T))$,  $\psi_2 \in C^1(\overline{\Omega_2} \times (0,T))$, such that $\psi=\psi_1$ in $\overline{\Omega_1} \times (0,T)$ and $\psi=\psi_2$ in $\overline{\Omega_2} \times (0,T)$. 

\begin{defi} \label{defHcont2} {\bf Flux Limited t-continuous (FL-tc)} \\
A locally bounded  function $u : \R \times (0,T) \ds \R$ is a  (FL-tc) subsolution (resp. supersolution) of  \eqref{eqHJmod2} in $\R \times (0,T)$  if  for any test function $\psi \in PC^1(\R \times (0,T))$ and any local maximum  (resp. minimum) point $(x_0,t_0)$  of $u^* - \psi$  (resp. $u_*- \psi$) in $\R \times (0,T)$ we have 
$$
\begin{array}{cl}
\mbox{ if } x_0=0 & \psi_t(0,t_0)+\max\{ A(t_0), \Hu^- (t_0,0,\psi_x(0^+,t_0) ), \Hd^+(t_0,0,\psi_x(0^-,t_0)) \} \leq 0   \quad \mbox{(resp. $\geq 0$)}  \\ 
 \\
\mbox{ if } x_0 \in \Omega_i & \psi_t(x_0,t_0)+ \Hi(t_0,x_0,\psi_x(x_0,t_0)) \leq 0  \quad  i=1,2 \:. \quad \mbox{(resp. $\geq 0$)} \\
\end{array}
$$
A locally bounded function $u$ is a (FL-tc) {\em solution} if it is both a sub and and a super-solution of  \eqref{eqHJmod2} and it verifies $u(x,0)= u_0(x)$ in $\R$.
\end{defi}
In the case of HJ equation in $\R^M$ with $t$-measurable Hamiltonians  solution have been firstly defined  by Ishii in 1985 (\cite{Is85}). We recall here this definition for equation 
\be \label{EqHJclassi}
\left\{
\begin{array}{lcl}
u_t +H(t, x, D_x u)=0 & \mbox{in} &  (0,T) \times \R^M \times \R^M  \\
u(x,0)= u_0(x)& \mbox{ in } &  \R^M 
\end{array}
\right.
\ee
where the Hamiltonian $H$ verifies regularity assumptions \eqref{ipHami2reg}-\eqref{ipHami2int}. In \cite{Is85} comparison results are obtained assuming that for any $R>0$, there exist a  $t$-measurable modulus $\omega_R$ such that 
\be \label{ipHamiIssGen}
| H(t,x,p)- H(t,y,q)|  \leq   \omega_R(t, |x-y|, |p-q|) \quad \forall |x|, |y|, |p|, |q| \leq R \:, t \in [0,T] . 
\ee
\vspace{0,2cm}
\begin{defi} \label{defHMis2} {\bf Ishii t-measurable (Is-tm)} \\
A locally bounded continuous function $u : \R^M \times (0,T) \ds \R$ is a  (Is-tm) subsolution (resp. supersolution)  of  \eqref{EqHJclassi}  in $(0,T)\times \R^M \times \R^M$ if: \\
if $b \in L^1(0,T)$, $G(t,x,p) \in C([0,T] \times \R^M \times \R^M) $   and $G$ satisfies for some $\delta >0 $, $(x_0, t_0) \in \R^M \times [0,T]$,  $\psi \in C^1( \R^M \times  (0,T))$

 $$
\begin{array}{c}
   -b(t)+G(t,x,p) \leq H(t,x, p)     \quad    (resp. \geq )  \\
   \\
\mbox{ for }  |x-x_0| \leq \delta \: , |p - D_x \psi(x_0,t_0) | \leq \delta  \mbox{ and a.e. }  |t - t_0| \leq \delta   \\
 \end{array}
$$
 and if $\dis u(x,t) -\big(  \int_0^t b(s) ds + \psi(x,t)\big)$   has a local maximum  (resp. minimum) point  in $\R^M \times (0,T)$ at $(x_0, t_0)$ then
$$
 \psi_t(x_0, t_0)+ G(t_0, x_0, D_x \psi(x_0,t_0) ) \leq 0 \quad  (resp. \geq 0) \: . 
$$
A locally bounded continuous function  $u$ is a {\em (Is-tm) solution}  of  \eqref{EqHJclassi}  if it is both a sub and a super-solution and  verifies $u(x,0)=u_0(x) $ in $\R$.
\end{defi}

With these two definitions in hand,  it is natural to define continuous  Flux Limited t-Measurable (FL-tm) solution as follows.  

\begin{defi} \label{defHmismod2} {\bf Flux Limited t-measurable (FL-tm)} \\
A continuous bounded  function $u : \R \times (0,T) \ds \R$ is a (FL-tm) subsolution (resp. supersolution) of \eqref{eqHJmod2} in $\R \times (0,T)$ if the following holds. 

Let   $b \in L^1(0,T)$, $\psi \in PC^1(\R \times (0,T))$ and $(x_0,t_0)$ be a local maximum (resp. minimum) point   in $\R \times (0,T)$ of 
$ \dis u(x,t) - \big( \int_0^t b(s) ds + \psi(x,t)\big)$.

{\bf Case 1:  If $ x_0 =0$ (on the junction).} Let  $G_0 \in C([0,T] \times \overline{\Omega_1} \times \overline{\Omega_2} \times \R^2)$ satisfies for some $\delta >0$, 
$$
\begin{array}{c}
 -b(t)+G_0(t,x_1,x_2,(p_1,p_2)) \leq \max\{ A(t), \Hu^+ (t,x_1,p_1), \Hd^-(t,x_2,p_2) \}       \quad    (resp. \geq )  \\
 \\
\mbox{ for $|p_1 - \psi_x(0^-,t_0) | \leq \delta$,  $|p_2 - \psi_x(0^+,t_0) | \leq  \delta$, $|x_1|,|x_2| \leq \delta$, and a.e.  $|t - t_0| \leq \delta$ } \\
 \end{array}
$$
then 
$$
 \begin{array}{c}
  \psi_t(0,t_0)+ G_0(t_0, 0,0, (\psi_x(0^-,t_0),\psi_x(0^+,t_0)) ) \leq 0  \quad    (resp. \geq  0) \: .  \\
 \end{array}
$$

{\bf Case 2:  If $ x_0 \in \Omega_i$ for $i=1,2$ (away from the junction).} Let $G_i \in C([0,T] \times \Omega_i   \times \R)$ satisfies for some $\delta >0$, 
$$
\begin{array}{c}
 -b(t)+G_i(t,x,p)  \leq \Hi(t,x,p) \quad    (resp. \geq )  \\
 \\
 \mbox{for $|p - \psi_x(x_0,t_0) | \leq \delta$,  $|x-x_0| \leq \delta$, and a.e.  $|t - t_0| \leq \delta$  }  \\
\end{array}
$$
then 
$$
 \begin{array}{c}
   \psi_t(x_0,t_0)+ G_i(t_0, x_0, \psi_x(x_0,t_0) ) \leq 0  \quad    (resp. \geq  0) \: .  \\
 \end{array}
$$

Moreover, a continuous bounded function $u$ is a {\em (FL-tm) solution}  of \eqref{eqHJmod2} if it is both a sub and a super-solution of \eqref{eqHJmod2} in $\R \times (0,T)$  and  verifies $u(x,0)=u_0(x)$ in $\R$.

\end{defi}

\begin{rem} \label{confaltre}
{\em 
It is clear that this definition generalizes both (Is-tm) and (FL-tc) definitions. Indeed, if there are not junctions, i.e. $H_1 \equiv H_2$, Definition \ref{defHmismod2} coincides  with  Definition \ref{defHMis2} while in the case of both $t$-continuous Hamiltonians and flux limiter, with $A(\cdot)\geq A_0(\cdot)$, we  recover Definition \ref{defHcont2}. }
\end{rem}

\begin{rem} \label{disdipinx}
{\em 
Comparing this definition with the condition introduced in the Introduction for the model problem  equation \eqref{eqHJmodIntro} we see that   when the test functions are touching at the junction we need to introduce the function $G_0$ that depends also on the space variables $ (x_1,x_2) \in  \overline{\Omega_1} \times \overline{\Omega_2}$. \\
This is linked to the  measurability $t$ dependence in Hamiltonians and not on the $x$ dependence as one can expect. Indeed, if one consider the case Hamiltonians depending only in space and gradient variable, i.e. $H_i=H_i(x,p)$,  it suffices to consider, in Definition  \ref{defHmismod2}, continuous $G_0=G_0(t,p_1,p_2)$ verifying  
$$
\begin{array}{c}
 -b(t)+G_0(t,(p_1,p_2)) \leq \max\{ A(t), \Hu^+ (0,p_1), \Hd^-(0,p_2) \}       \quad    (resp. \geq )  \\
 \end{array}
$$
for $|p_1 - \psi_x(0^-,t_0) | \leq \delta$,  $|p_2 - \psi_x(0^+,t_0) | \leq  \delta$ and a.e.  $|t - t_0| \leq \delta$.    }

\end{rem}
\section{On comparison results} \label{comp2} 

With the aim of giving a comparison result that can be easily extended to more general Hamiltonians and geometries,  we give a proof based on the same ideas as in  \cite{BrRam}  by the author and  F. Rampazzo (this approach  is inspired by the proof of the stability result  proved for (Is-tm)  solutions in the original work \cite[Proposition 7.1]{Is85}).

{\em   

As already done in the introduction, we first explain the key construction in the simple framework of the ODE: 
 \be \label{ODEM}
 u^\prime(t)+ H(t)=0  \mbox{ for } t \in (0,T) \mbox{ with }  H(t) \in L^1(0,T). 
 \ee
Suppose  there exists a sequence of continuous functions $H_n  \in C^0(0,T)$ such that $H_n \ds H$ in $L^1(0,T)$ and  therefore that comparison results are classically known for the corresponding ODE.

The key idea is that if  $u$ is a (Is-tm)-subsolution as in Definition \ref{defHMis2},  we can construct a sequence $(u_n)$ of classical viscosity subsolution of $u^\prime(t)+ H_n(t)=0$ that converges to $u$ pointwise,  by setting 
$$u_n(t):=u(t)-\int_0^t |H_n(s)-H(s)| \, ds. $$ 
Indeed,  for any test function  $\psi \in C^1(0,T)$ and  $t_0$ point of maximum of $u_n-\psi$ in $(0,T)$,  $t_0$ is also a maximum point of $u(t)- \big(\int_0^t |H_n(s)-H(s)| \, ds +\psi(t) \big)$.  Thus applying Definition \ref{defHMis2}  with $b(s)=|H_n(s)-H(s)| $ and $G(t)=H_n(t)$ we obtain that  $\psi^\prime(t_0)+ H_n(t_0) \leq 0$.  \\
With similar arguments, if $v$ is a (Is-tm)-supersolution as in Definition \ref{defHMis2},  we can construct  $$v_n(t):=v(t)+\int_0^t |H_n(s)-H(s)| \, ds$$ that is a  classical viscosity supersolution of $v^\prime(t)+ H_n(t)=0$ in $(0,T)$. \\
Since by construction $u_n \ds u$ and $v_n \ds v$,  it is clear that one can obtain comparison results for the $t$- measurable case once a comparison result is know for the continuous case.  }

Exploiting this  idea our result will be then based on one hand on  the corresponding known results in the cases of Flux Limited $t$-continuous solutions and on the other on the possibility to construct a {\it good} approximation of the Hamiltonians and the flux limiter $A$.  Roughly speaking, we will need to approximate by a sequences of  $t$-continuous functions converging in $L^1(0,T)$ somehow uniformly w.r.t. the other variables (see   \eqref{ipHnLOC}-\eqref{ipAn2LOC} in the proof of Theorem  \ref{compmod2LOC} and Remark \ref{gendim1}). As discussed in \cite[Section 4.1]{BrRam} a quite general assumption that guarantees such an approximation is the following: \\
{\em (AP)} For every compact subset  $Q \subset \R^2$  there exist $h \in L^\infty(0,T)$ and a  $t$-integrable modulus $\omega$  such that,  for  $i=1,2$, 
\be  \label{ipHam2inTmodulus} 
| H_i(t,x,p) - H_i(s,y,q)| \leq | h(s) - h(t)| + \omega(t, |x-y|+|p-q|) + \omega(s,|x-y|+|p-q|) \: 
\ee
for any $(x,p), (y,q) \in Q$ and $t,s \in [0,T]$. 

Our comparison result, Theorem \ref{compmod2LOC} below, is proved under the assumption of Section \ref{onedimconvex}. This is motivated by the aim of applying the corresponding comparison result in  \cite[Theorem 5.8]{ImMo1} for (FL-tc) solutions. Indeed, if we assume $A$ and  $H_i$ to be $t$-continuous,  
 \eqref{ipdatoiniziale2}-\eqref{ipA2} trivially imply (H0)-(H6) and (A0)-(A2) in  \cite[Section 5.3]{ImMo1}. Moreover,  one can verify that \eqref{ipdatoiniziale2}-\eqref{ipA2} also implies assumptions {(GA-Conv)-(GA-G-FL)}  in the more general comparison  result given in \cite[Theorem 14.2]{BaChas}, for $t$- continuous case Hamiltonians and Flux-limiter. 

We observe also that  away from the junction  $\{0 \}$ each Hamiltonian verifies assumption  \eqref{ipHamiIssGen}, therefore, without junctions,  we are in the framework of comparison results  given in \cite[Theorem 8.1]{Is85} or  \cite[Appendix]{LiPer87}.

 \begin{theo}  {\bf (Comparison result)}  \label{compmod2LOC} \\
Assume the Hamiltonians $H_i $ satisfies the regularity assumptions  \eqref{ipHami2reg}-\eqref{ipHam2inP}, 
 the convexity hypothesis \eqref{ipHami2Conv} and fulfill the approximability condition \eqref{ipHam2inTmodulus}. 
Let $A$ be a flux limiter verifying  \eqref{ipA02}-\eqref{ipA2}. \\
Let $u$ and $v$ be a (FL-tm)-subsolution and a (FL-tm)-supersolution of \eqref{eqHJmod2}, respectively.  
Let us assume that they are locally Lipschitz continuous in $x$ uniformly w.r.t. $t \in [0,T]$, and that  there exist $C_1,C_2 >0$ such that 
$$
u(x,t)  \leq  C_1(1 + |x|)  \:, \quad  v(x,t)  \geq  C_2(1 + |x|) \quad  \forall (x,t) \in \R \times (0,T) \: .
$$
If 
$$
u(x,0) \leq v(x,0)  \mbox{ for all }  x \in \R. 
$$
Then, 
$$
 u(x,t) \leq v(x,t) \mbox{ for all } (x,t) \in \R \times (0,T).
$$
\end{theo}
{\bf Proof.}  This proof is divided 3 steps: we start by localizing the proof,  we then construct the sequences of approximating sub and super solutions in a ball $B_R=(-R,R)$ and  we end by detailing how to exploit the  corresponding result in the continuous case to conclude. 

{\bf Step 1 (localization).} 
Set, for  $\eta,\alpha >0$ 
$$
w_{\alpha,\eta}(x,t):= \tilde{u}_{\alpha,\eta}(x,t)-v(x,t) \quad \mbox{ where  } \quad 
\tilde{u}_{\alpha,\eta}(x,t):= u(x,t)-\eta t- \frac{\alpha}{2} |x|^2   \:.
$$
Arguing similar to classical framework (see e.g. \cite[Section 2.2]{BaChas} or \cite[proof of Theorem 1.5]{ImMo1}),  we can prove that $\exists R=R(\alpha,\eta) >0 $ such that 
\be \label{DISfuoriBR}
  \vert x \vert \geq R \Longrightarrow  w_{\alpha,\eta}(x,t) < 0 \quad \forall t \in (0,T), 
\ee
and that  fix $\eta >0$ there exists a small enough $\alpha= \alpha(\eta) >0$ such that  $\tilde{u}_{\alpha,\eta}$ is a  (FL-tm)-subsolution of \eqref{eqHJmod2}, for any  $ \alpha \leq \alpha(\eta) $. Thus, we have 
\be  \label{locUV}
 \tilde{u}_{\alpha,\eta}(x,t) <  v(x,t) \: \mbox{ on } \partial B_R \times (0,T)  \quad \mbox{ for } \eta >0 \:, 0< \alpha \leq  \alpha(\eta)  \mbox{ and } R=R(\alpha(\eta), \eta) .
 \ee
Moreover,  by construction,  $\tilde{u}_{\alpha,\eta}(x,0) < v(x,0)$ in $\R$.

{\bf Step 2 (construction of $u_n$ and $v_n$).} 
Fix, now $R$ as in \eqref{locUV} and  $K=\max\{L_u,L_v\}$, with $L_u, L_v$  being  the Lipschitz constants of $u$ and $v$  in $B_R$, respectively.  Thanks to assumption \eqref{ipHam2inTmodulus}, there exist $(\Hin)_{n \in \N} \in C^0([0,T] \times \overline{\Omega}_i \times \R ; \R)$, $i=1,2$  such that 
\be \label{ipHnLOC} 
\begin{array}{l}
\dis  \lim_{n \ds \infty}  \int_0^T  \sup_{ x \in  \overline{B_{R,i}} ,  \/ |p| \leq K}  | \Hin(t,x,p)- H_i(t,x,p)| dt =0
 \end{array} 
\ee
 where we set $B_{R,i}=B_R \cap \Omega_i$.
Moreover, by hypothesis \eqref{ipA2} there exists  a sequence of functions $(A_n)_{n \in \N}$,  such that 
\be \label{ipAn2LOC} 
\begin{array}{l}
\dis  A_n \in C^0([0,T]; \R),   \quad A_n \rightarrow A \mbox{ in }  L^1(0,T). \\ 
 \end{array} 
\ee
Let
\be \label{defknLOC}
k_n(t):= \max \{ k^0_n(t), k^1_n(t),k^2_n(t) \}
 \ee
where
 \be \label{defknoudLOC}
 \begin{array}{l}
\dis k^0_n(t):= \sup_{|p_1| ,|p_2| \leq K} \left| \max\{ A(t), \Hu^+ (t,0,p_1), \Hd^-(t,0,p_2) \} -   \max\{ A_n(t), \Hun^+ (t,0,p_1), \Hdn^-(t,0,p_2) \} \right|  \\
\\
\dis k^i_n(t):= \sup_{ x \in  \overline{B_{R,i}}, \: |p| \leq K} \left| \Hi (t,x,p)  -  \Hin (t,x,p) \right|  \: ,\quad i=1,2
\end{array}
\ee
 and define 
$$
 u_n(x,t):=\tilde{u}_{\alpha,\eta}(x,t)-\int_0^t k_n(s) \: ds  \mbox{ and }  v_n(x,t):= v(x,t)+\int_0^t k_n(s) \: ds.
$$
Key points of this construction are:  $k_n$ is a sequence of $L^1$-functions such that  $k_n \rightarrow 0 $ in $L^1(0,T)$ and,  for any test function $\psi \in PC^1(\R \times (0,T))$ touching $u_n$ or $v_n$ from above or from below at  $(0,t)$ (resp. $(x,t)$, $x \in   \overline{B_{R}}$) one has $\max  \{ |\psi_x(0^+,t)|, |\psi_x(0^-,t)| \} \leq K$, (resp. $|\psi_x(x,t)|  \leq K$).

{\bf Step 2-(i).}  For every $n \in \N$,  $u_n$ is a  (FL-tc) subsolution of \eqref{eqHJmod2}$_n$ in  $B_{R} \times (0,T)$. (We will refer to equation \eqref{eqHJmod2}$_n$ when we will consider equation \eqref{eqHJmod2} with $A_n$, $\Hin$  instead of $A$, $H_i$, respectively.)

Let $\psi \in PC^1(\R \times (0,T))$ and   $(x,t)$ be a  local maximum point of $u_n - \psi$ in $B_{R}  \times (0,T)$. 
By construction $(x,t)$ is a  local maximum point of 
$$
\tilde{u}_{\alpha,\eta}(x,t)-( \int_0^t k_n(s)) \: ds+  \psi(x,t) ) \mbox{ with } k_n \in L^1(0,T).
$$

Let us consider two cases depending on  $x$ being on the junction or not. 

If $x \in \Omega_i$,  we set  $G_n(t,x,p) := \Hin(t,x,p)$.
Thus $G_n \in C([0,T] \times \overline{\Omega}_i \times \R)$ and 
$$
-k_n(t)+G_n(t,x, p) \leq \Hi(t,x,p) \quad \mbox{ for any }  x \in  \overline{B_{R,i}} \:,  |p|  \leq  K \: .
$$
Since $\tilde{u}_{\alpha,\eta}$ is a (FL-tm) subsolution of \eqref{eqHJmod2}  and $|\psi_x(x,t)| \leq  K$, Definition  \ref{defHmismod2} applies and gives 
$$
 \psi_t(x,t)+ G_n(t,x, \psi_x(x,t))  \leq 0,
$$
that reads 
$$
  \psi_t(x,t)+ \Hin(t,x,\psi_x(x,t)) \leq 0  
$$
thus Definition  \ref{defHcont2} is fulfilled. 

If $x=0$  setting $G_n(t,x_1,x_2, (p_1,p_2)):= \max\{ A_n(t), \Hun^+ (t,x_1,p_1), \Hdn^-(t,x_2,p_2) \}$, one has  \\ $G_n \in C([0,T] \times \overline{\Omega_i} \times \overline{\Omega_i}   \times \R^2)$,   and for any  $|p_1|,|p_2|  \leq  K$,  $t \in (0,T)$ 
$$
-k_n(t)+G_n(t,0,0, (p_1,p_2)) \leq \max\{ A(t), \Hu^+ (t,0,p_1), \Hd^-(t,0,p_2) \}.
$$
Since $\tilde{u}_{\alpha,\eta}$ is a (FL-tm) subsolution of \eqref{eqHJmod2} and $|\psi_x(0^+,t)|, |\psi_x(0^-,t)| \leq  K$, Definition  \ref{defHmismod2} applies and gives 
$$
 \psi_t(0,t)+ G_n(t,0,0, (\psi_x(0^-,t),\psi_x(0^+,t)))  \leq 0, 
$$
that reads 
$$
 \psi_t(0,t)+\max\{ A_n(t), \Hun^- (t,0,\psi_x(0^+,t) ), \Hdn^+(t,0,\psi_x(0^-,t)) \} \leq 0
$$
and conclude Step 2-(i).

{\bf Step  2-(ii)} For each $n$ fixed,  $v_n$ is a  (FL-tc) supersolution of \eqref{eqHJmod2}$_n$ in $B_{R} \times (0,T)$.

Let $\psi \in PC^1(\R \times (0,T))$ and   $(x,t)$ be a  local minimum point of $v_n - \psi$ in $\R \times (0,T)$. 

By construction $(x,t)$ is a  local minimum point of 
$$
v(x,t)-  \Big( \int_0^t (-k_n(s)) \: ds +  \psi(x,t) \Big) \mbox{ with } k_n \in L^1(0,T) .
$$
Let us consider two cases depending on  $x$ being on the junction or not.

If $x \in \Omega_i$,  we set  $G_n(t,x,p) := \Hin(t,x,p)$.
Thus $G_n \in C([0,T] \times \overline{\Omega}_i \times \R)$ and 
$$
-k_n(t)+G_n(t,x, p) \geq \Hi(t,x,p) \quad \mbox{ for any }  x \in  \overline{B_{R,i}} \:,  |p|  \leq  K \: .
$$
Since $v$ is a (FL-tm) supersolution of \eqref{eqHJmod2} and $|\psi_x(x,t)| \leq  K$, Definition  \ref{defHmismod2} applies and gives
$$
 \psi_t(x,t)+ G_n(t,x, \psi_x(x,t))  \geq 0,
$$
that reads 
$$
  \psi_t(x,t)+ \Hin(t,x,\psi_x(x,t)) \geq 0  
$$
thus Definition  \ref{defHcont2} is fulfilled. 

If $x=0$  setting $G_n(t,x_1,x_2, (p_1,p_2)):= \max\{ A_n(t), \Hun^+ (t,x_1,p_1), \Hdn^-(t,x_2,p_2) \}$, one has  \\ $G_n \in C([0,T] \times \overline{\Omega_i} \times \overline{\Omega_i}   \times \R^2)$,   and for any  $|p_1|,|p_2|  \leq  K$,  $t \in (0,T)$ 
$$
k_n(t)+G_n(t,0,0, (p_1,p_2)) \geq \max\{ A(t), \Hu^+ (t,0,p_1), \Hd^-(t,0,p_2) \} . 
$$
Since $v$ is a (FL-tm) subsolution of \eqref{eqHJmod2} and $|\psi_x(0^+,t)|, |\psi_x(0^-,t)| \leq  K$, Definition  \ref{defHmismod2} applies and gives 
$$
 \psi_t(0,t)+ G_n(t,0,0, (\psi_x(0^-,t),\psi_x(0^+,t)))  \geq 0, 
$$
that reads 
$$
 \psi_t(0,t)+\max\{ A_n(t), \Hun^- (t,0,\psi_x(0^+,t) ), \Hdn^+(t,0,\psi_x(0^-,t)) \} \geq 0
$$
and conclude Step 2. 

{\bf Conclusion.}  
By construction, we have 
$$
u_n(x,0)= \tilde{u}_{\alpha,\eta}(x,0) \leq u(x,0) \leq v(x,0)=v_n(x,0) \quad \forall x \in \R \:. 
$$
and,  
\be  \label{limconst}
\lim_{n \ds \infty} \sup_{(x,t) \in \R \times (0,T)} |  u_n(x,t)-\tilde{u}_{\alpha,\eta} (x,t)  |=0 \quad \mbox{and} \quad \lim_{n \ds \infty} \sup_{(x,t) \in \R \times (0,T)} |  v_n(x,t)-v(x,t)  |=0  \: .
\ee
Therefore,  by \eqref{locUV} in Step 1, for $n$ big enough 
$$
 u_n(x,t) \leq   v_n(x,t) \: \mbox{ on } \partial B_R \times (0,T)   \: .
$$
Moreover,  \eqref{ipHami2reg}-\eqref{ipA2}  imply (H0)-(H6) and (A0)-(A2) in  \cite{ImMo1} (or assumptions {(GA-Conv)-(GA-G-FL)})  in \cite{BaChas}, thus  the comparison results for (FL-tc) solutions  \cite[Theorem 5.8]{ImMo1} (or \cite[Theorem 14.2]{BaChas}) apply and  we have 
$$
u_n(x,t) \leq v_n(x,t) \mbox{ for all }  (x,t) \in B_R \times (0,T).
$$
Letting $n$ tends to infinity, and recalling  \eqref{DISfuoriBR}   we obtain
$$
 \tilde{u}_{\alpha,\eta}(x,t) \leq   v(x,t) \: \mbox{ in }  \R \times (0,T)  \: ,
$$
conclusion will then follows letting $\alpha$, and then $\eta$, tending to zero. 
 \hfill $\Box$

\begin{rem} \label{easycase}
{\em
Let us point out here that the construction of the  approximating sequence $(\Hin)$ is not needed when dealing with the model problem, equation \eqref{eqHJmodIntro}. 
Indeed, the same argument as above works if one choses $k_n(t):=k^0_n(t)$ in the construction of $u_n$ and $v_n$, (with  $k^0_n(t)$ as in  \eqref{defknoudLOC} and  $(A_n)_{n \in \N}$ as in  \eqref{ipAn2LOC}),  and take 
$$
G_n(t,p_1,p_2):=\max\{ A_n(t), \Hu^+(p_1), \Hd^{-}(p_2) \}
$$ 
when touching at the junction in the proof of $u_n$ and $v_n$ being (FL-tc) sub and supersolution, respectively. \\
Of course the same remark holds if $H_i=H_i(t,x,p)$ are assumed to be continuous in time.
  }
\end{rem}

\begin{rem}  \label{gendim1}
{\em Also with the aim to discuss as these results could be adapted to more general framework, we observe here that we could have stated the comparison principle under less general assumptions  but allowing a global proof, i.e. without Step 1.  
The idea is to be in a situation where the approximating sequences $u_n$ and $v_n$ are, respectively,  sub and super (FL-tc)-solutions of   \eqref{eqHJmod2}$_n$  in all $\R \times (0,T)$. For instance,  one can  assume one of  the following conditions.  
\begin{itemize}
 \item[(a)] Let $u$ and $v$ be a (FL-tm)-subsolution and a (FL-tm)-supersolution of \eqref{eqHJmod2}, respectively.  Let us assume that they are {\it globally Lipschitz continuous} in $x$ uniformly w.r.t. $t \in [0,T]$. \\
Moreover,  there exists a sequence of continuous Hamiltonians 
$(\Hin)_{n \in \N} \in C^0([0,T] \times \Omega_i \times \R ; \R)$, $i=1,2$  such that 
\be \label{ipHn} 
\begin{array}{l}
\dis  \lim_{n \ds \infty}  \int_0^T  \sup_{x \in \R ,  |p| \leq K}  | \Hin(t,x,p)- H_i(t,x,p)| dt =0.
 \end{array} 
\ee
 \item[(b)] Let $u$ and $v$ be a (FL-tm)-subsolution and a (FL-tm)-supersolution of \eqref{eqHJmod2}, respectively.  Let us assume that they are  bounded and continuous  in  $\R \times   (0,T)$. \\ 
Moreover,  there exists a sequence of continuous Hamiltonians $(\Hin)_{n \in \N} \in C^0([0,T] \times \Omega_i \times \R ; \R)$, $i=1,2$,  such that 
\be \label{ipHnGLOB} 
  \lim_{n \ds \infty}  \int_0^T  \sup_{x \in \R, p \in \R}  | \Hin(t,x,p)- H_i(t,x,p)| dt =0.
\ee
\end{itemize}
Indeed, in both cases  the suprema  in  \eqref{defknoudLOC}  to define $k_n$ in \eqref{defknLOC}, can be taken independently of $R$ and/or $K$ and the proof can be done globally.
}
\end{rem}


\section{An existence result  via optimal control problems} \label{submodoc2} 

The aim of this section is to prove an existence result considering an optimal control problem and  proving that the corresponding value function is a (FL-tm) solution  of equation \eqref{eqHJmod2} as in  Definition  \ref{defHmismod2}. The corresponding construction in the $t$-continuous case is described in  \cite[Section 14.4]{BaChas}, (see also \cite{BBCI}).

 Fix $T > 0$. For $i=0,1,2$,  let $\A_i$ be compact subsets of $\R$. We set $\A:=\A_0 \times \A_1 \times \A_2$ and consider as controls the triples $\alp= (\al0, \alu, \ald) \in L^\infty ([0,T] ; \A)$.  
 
Dynamics $\bi :[0,T] \times \R \times \A_i \ds \R$, $i=1,2$,  satisfies: 
\be \label{ipRC2Regdin}
\bi(t,x,\alpha_i) \mbox{ are  bounded, measurable in } t   \mbox{ and continuous in } (x, \alpha_i) 
\ee
and, for each $R >0$, there exists $C=C(R) >0$ such that 
\be \label{pRC2Regdin2}
 | \bi(t,x,\alpha_i)- \bi(t,y, \alpha_i) | \leq C(R) |x-y| \quad \forall |x|,|y| \leq R, \: t \in [0,T], \: \alpha_i \in  \A_i \: .
\ee

Running costs $\li :[0,T] \times \R \times \A_i \ds \R$ , $i=1,2$,  satisfies: 
\be \label{ipRC2Regrcos}
\li(t,x,\alpha_i) \mbox{ are  bounded,  measurable in  } t   \mbox{ continuous in } (x, \alpha_i) 
\ee
and, for each $R>0$ there exists a $t$-measurable modulus such that $\omega_R$ such that 
\be \label{ipRCreg}
| \li(t,x,\alpha_i) -l_i(t,y,\alpha_i) | \leq \omega_R (t,|x-y|) \quad  \forall |x|,|y| \leq R, \: t \in [0,T], \: \alpha_i \in  \A_i  \: . 
\ee

Moreover, for each $(t,x) \in [0,T] \times \R^N$, the sets 
\be \label{ipRC2}
 \{ ( \bi(t,x,\alpha_i), l_i(t,x,\alpha_i) ) \: : \: \alpha_i \in  \A_i \} \mbox{  are compact and convex }
\ee
and we  assume normal controllability in $\{0\}$:    for some $\delta >0$  independent of $(x,t)$  the sets 
\be  \label{IpControl22}
\{  \bu(t,x,\alpha_1)  \: : \: \alpha_1 \in  \A_1 \}  \mbox{ and } \{ \bd(t,x,\alpha_2)  \: : \: \alpha_2 \in  \A_2 \} \mbox{ contains }   [ -\delta, \delta] 
\ee
and, in order to allow admissible trajectories to stay in $\{0\}$, we ask that 
\be \label{ipdinamicain0}
  0 \in \A_0  \: .
\ee

{\it Admissible trajectories.} An admissible trajectory $\gamma$ that goes from $\ga0$ at time $0$ to $x$ at time $t$  is such that  there exists  a global control $\alp= (\alu, \ald, \al0) \in L^\infty ([0,t] ; \A)$,  and  a partition of $(0,t)$, $I=(\Iu, \Id, \I0) $ of measurable sets such that $\gamma(s) \in \Omega_i$ if $s \in I_i$, $i=1,2$, $\gamma(s)=0$ if $s \in I_0$ and 
 \be \label{deftra2}
 \dot{\gamma}(s)= \sum_{i=1,2} \bi(s,\gamma(s),\ali(s)) \mathbf{I}_{\Ii} (s) + \al0(s)  \mathbf{I}_{I_0}(s) \mbox{ for almost every } s \in (0,t).
\ee
Note that by Stampacchia's theorem (see for instance \cite{GiTr}) $\dot{\gamma}(s)=0$ for almost every $s \in I_0$, therefore \eqref{deftra2} reads 
 $$
 \dot{\gamma}(s)= \sum_{i=1,2}  \bi(s,\gamma(s),\ali(s))  \mathbf{I}_{\Ii} (s)  \mbox{ for almost every } s \in (0,t).
$$
We will denote the set of admissible trajectories that goes from $\gamma_0$ at time $0$ to  $x$ at time $t$ by
$$
{\T}_{0,\gamma_0}^{t,x}:= \apg (\gamma(\cdot), \alpha(\cdot)) \in \mbox{Lip}([0,t];\R) \times L^\infty((0,t); \A) : \eqref{deftra2}  \mbox{ holds, } \:  \gamma(0)=\gamma_0, \gamma(t)=x \: \chg  .
$$

{\it The cost function.} Let $l_0 \in L^\infty([0,T])$  and $A_0 \in \R$ we set 
\be \label{ipA0con}
A(t):=\max(-l_0(t),A_0) \in L^\infty([0,T]; [A_0,|A_0|+\parallel l_0 \parallel_\infty] ).
\ee

Let $u_0 : \R \ds \R$, for any $(\gamma,\alpha) \in {\T}_{0,\gamma_0}^{t,x}$ we consider the cost 
$$
J((x,t); (\gamma,\alpha)):=u_0(\gamma_0)+ \int_0^t( l_1(s,\gamma(s),\alu(s)) \mathbf{I}_{\Iu} (s) +l_2(s,\gamma(s),\alu(s)) \mathbf{I}_{\Id} (s) -A(s) \mathbf{I}_{I_0} (s)) \: ds \: .
$$

\begin{defi} \label{VF}  {\em 
Let $u_0 : \R \ds \R$ be bounded and  Lipschitz continuous. For $(x,t) \in \R \times [0,T]$,  the {\em value function} is 
$$
u(x,t) : = \inf_{\gamma_0 \in \R} \:  \inf_{\gamma \in  {\T}_{0,\gamma_0}^{t,x}}   J((x,t) ; (\gamma, \alpha))=  \inf_{\gamma_0 \in \R} \:  \inf_{\gamma \in  {\T}_{0,\gamma_0}^{t,x}}  
\apg u_0(\gamma_0)+ \int_0^t L(\tau,\gamma(\tau),\alpha(\tau)) \: d \tau \chg
 $$
 where we set 
 \[
L(t,x,\alp):= \left\{
\begin{array}{lcl}
l_1(t,x,\alu(t)) & \mbox{if} & x \in  \Omega_1    \\
l_2(t,x,\ald(t)) & \mbox{if}  &  x \in \Omega_2  \\
-A(t) & \mbox{ if}  &  x=0 \:. \\
\end{array}
\right.
\]}
\end{defi}

It is easily seen that  a standard {\em Dynamic Programming Principle (DPP)}   holds: 
for each  $(x,t) \in \R \times [0,T]$, $0 < s < t$, 
\be \label{DPP}
u(x,t)=   \inf_{z \in \R} \: \inf_{(\gamma, \alpha) \in  {\T}_{s,z}^{t,x}}  \apg  u(z,s) + \int_s ^t L(\tau,\gamma(\tau),\alpha(\tau)) \: d \tau \chg \: .
\ee

With similar arguments as in the continuous case we can derive from the DPP a regularity result for the value function. Precisely:

\begin{pro}   \label{regvalOnedim}  
Assume \eqref{ipRC2Regdin}, \eqref{pRC2Regdin2}, \eqref{ipRC2Regrcos}, \eqref{ipRCreg}, \eqref{ipRC2}, \eqref{IpControl22}, \eqref{ipdinamicain0} and \eqref{ipA0con}. \\
Then, the value function in Definition  \ref{VF}  is a bounded Lipschitz continuous function in $\R \times (0,T)$.
\end{pro}
{\bf Proof.}  
One easily has that  the value function is bounded, i.e. 
$$
\vert u(x,t)   \vert  \leq (2L+\bar{A}) T + \parallel u_0 \parallel_\infty   \mbox{ for all }  (x,t) \in \R \times (0,T) \:,
$$
where $\dis L:=\max(   \parallel  l_1 \parallel_\infty, \parallel l_2 \parallel_\infty )$ and $\bar{A}:=|A_0|+\parallel l_0 \parallel_\infty $.

Moreover, since controllability assumptions allow  constant trajectories, by DPP  
$$
\vert u(x,t) - u(x,s) \vert \leq  (2L+\bar{A}) \vert t-s \vert \mbox{ for all }  x \in \R \: ,  t,s \in  (0,T).
$$

Let us fix now any $x,y \in \R, x \not = y$, and let $\eta=:=\frac{y-x}{\vert y-x \vert} \delta$, with $\delta >0$ given in  \eqref{IpControl22}. \\ For $t  \in [0,T]$ we  take $s:= t- \frac{\vert y-x \vert}{\delta} < t$ and we  aim to construct an admissible trajectory that goes  from $y$ at time $s$ to $x$ at time $t$. If $x,y$ are both in the same $\Omega_i$ this is standard.  If this is not the case,  the construction is possible thanks to the choice of $\eta$, the normal controllability assumption \eqref{IpControl22} and a measurable selection theorem (e.g. Corollary of Proposition 1 in \cite{Roc}). \\
Therefore, using this trajectory in the DPP  we have:
$$
u(x,t)-u(y,t) \leq  u(x,t) - u(y,s)+u(y,s)-u(y,t) \leq \int_s^t L(\tau, \gamma(\tau), a) d\tau +u(y,s)-u(y,t) 
$$
thus 
$$
u(x,t)-u(y,t) \leq (2L+\bar{A}) \vert t-s \vert + (2L+\bar{A}) \vert t-s \vert = 2(2L+\bar{A}) \frac{\vert y-x \vert}{\delta}
$$
inverting the roles of $x$ and $y$ one can conclude that $u$ is Lipschitz continuous. 
\hfill $\Box$

\

We define the following  {\em Hamiltonians} 
\be \label{defH}
H_i(t,x,p):= \sup_{ \ali \in \A_i} \apg \bi(t,x,\alpha_i)  p - l_i(t,x,\ali)   \chg \mbox{ for } i=1,2 \: . 
\ee
Standard arguments ensures that both $H_i$ fulfill regularity conditions \eqref{ipHami2reg}-\eqref{ipHam2inP} and  are convex and coercive in $p$.  Moreover, also thanks to normal controllability,   their nondecreasing and nonincreasing part  can be characterized in terms of subset of controls.  The precise result is the following (we will omit the proof because the $t$- measurability plays not real role here, thus this is analogous to the continuous case, see e.g. \cite[Appendix]{BBCI},  \cite[Lemma 6.2]{ImMo1} or \cite[Section 14.4]{BaChas}). 

\begin{lem}  \label{lempropH}
Assume  \eqref{ipRC2Regdin}-\eqref{IpControl22}.
Then,  the Hamiltonians $H_i$ in \eqref{defH} fulfill  \eqref{ipHami2reg}-\eqref{ipHami2Conv}. \\
Moreover, the nondecreasing and nonincreasing part of $H_i$ are given by 
\begin{eqnarray*}
H_i^-(t,x,p)= \sup_{ \ali \in \A_i ,  \bi(t,x,\alpha_i) \leq 0} \apg \bi(t,x,\alpha_i) p - l_i(t,x,\ali)   \chg   \\
=\sup_{ \ali \in \A_i ,  \bi(t,x,\alpha_i) < 0} \apg \bi(t,x,\alpha_i) p - l_i(t,x,\ali)   \chg  
\end{eqnarray*}
and  
\begin{eqnarray*}
 \quad H_i^+(t,x,p)= \sup_{ \ali \in \A_i ,  \bi(t,x,\alpha_i)  \geq 0} \apg  \bi(t,x,\alpha_i) p - l_i(t,x,\ali)   \chg \\
 = \sup_{ \ali \in \A_i ,  \bi(t,x,\alpha_i)  > 0} \apg  \bi(t,x,\alpha_i) p - l_i(t,x,\ali)   \chg. 
\end{eqnarray*}
\end{lem}

We are finally ready to prove our existence result: in Theorem \ref{teoexismod2} below we prove that the value function in Definition  \ref{VF} is a  (FL-tm) solution of \eqref{eqHJmod2} in the sense of Definition \ref{defHmismod2}. 
\begin{rem} {\em 
Here we highlight two particular cases that we consider interesting. 
\begin{itemize} 
\item[(i)]  If dynamics and running costs  are assumed to be independent of $x$ and $t$, i.e. $\bi(t,x,\alpha_i)=\li(t,x,\alpha_i)=\alpha_i$, $i=1,2$ we are in the framework of the optimal control problem studied in \cite{CaSoug}. Thus, in  Theorem \ref{teoexismod2} we prove in particular that  map  like $u^A$  in \cite[Definition 2.1]{CaSoug}  are (FL-tm)  solution of equation \eqref{eqHJmodIntro}.
\item[(ii)]  If there is no junction, i.e. $\bu \equiv \bd$,  $l_1 \equiv l_2$ and $A \equiv 0$  we are in the framework classical (HJ)-equation  in the $t$-measurable case. Therefore,  
in  Theorem \ref{teoexismod2} below the proof  of Case 1 can be seen as the proof the  value function being  a (Is-tm) subsolution in Step 1 and a (Is-tm) supersolution in Step 2, respectively.  
At the best of our knowledge this result is not present in the literature. 
\end{itemize}
}
\end{rem}

\begin{theo} \label{teoexismod2}
Assume  \eqref{ipRC2Regdin}, \eqref{pRC2Regdin2}, \eqref{ipRC2Regrcos}, \eqref{ipRCreg}, \eqref{ipRC2}, \eqref{IpControl22}, \eqref{ipdinamicain0} and \eqref{ipA0con}.  Then,  the value function $u$ in Definition  \ref{VF} is a (FL-tm) viscosity solution of \eqref{eqHJmod2} in $\R \times (0,T)$ as in Definition \ref{defHmismod2}.
\end{theo}

{\bf Proof.} Let us  first remark that by construction and Proposition  \ref{regvalOnedim}  the value function $u$ is a continuous bounded function such that $u(x,0)=u_0(x)$ for any  $x \in \R$. 
Before detailing  the proof of $u$ being both a (FL-tm) subsolution and a (FL-tm) supersolution of \eqref{eqHJmod2}  in Steps 1  and 2 below, we 
derive a  useful identity.   

Consider any admissible trajectory arriving at time $t_0$ at the junction  $\{ 0 \}$, i.e. $(\gamma, \alpha) \in {\T}_{\tau,\gamma(\tau)}^{t_0,0} $ with $\tau \in [0,t_0)$. 
For any  $\psi=(\psi_1,\psi_2) \in PC^1(\R \times (0,T))$ we have 
\be \label{stimatestA2}
\begin{array}{l}
\dis \psi(0,t_0)-\psi(\gamma(\tau),\tau)-\int_\tau^{t_0}   L(s,\gamma(s), \alpha) \: ds =  \dis  o(t_0-\tau)+\psi_t(0,t_0)(t_0-\tau)+  \\
 \dis  \quad  \quad  \quad+\int_\tau^{t_0}  \big( \sum_{i=1,2} \left[ (\psi_i)_x(0,t) \bi(s, \gamma(s),\alpha_i(s))  - l_i( s,\gamma(s),\alpha_i(s))  \right] \mathbf{I}_{\Ii} (s) + A(s)  \mathbf{I}_{\I0} (s) \big) ds  \: .
\end{array}
\ee  
where we set $(\psi_1)_t (0,t)=(\psi_2)_t (0,t):=\psi_t(0,t)$ and we used the boundedness of  $\bi, \li$ in \eqref{ipRC2Regdin}, \eqref{ipRC2Regrcos}.

\

{\bf Step 1 (Subsolution). }   
Let   $b \in L^1(0,T)$, $\psi \in PC^1(\R \times (0,T))$ and $(x_0,t_0)$ be a local  maximum point   in $\R \times (0,T)$ of 
$ \dis    u(x,t) - \big( \int_0^t b(s) ds + \psi(x,t)\big)$.  \\
Consider  $0 < \tau < t_0$ and any admissible trajectory $(\gamma, \alpha) \in  {\T}_{\tau,\gamma(\tau)}^{t_0,x_0}$. \\
Thanks to the DPP, there exists  $\eta >0$ such that  for  $\tau < t_0$,$\vert \tau - t_0 \vert \leq \eta \:, \: \vert  \gamma(\tau)- x_0\vert \leq \eta$  we have 
 \be \label{subALLOmega1}
\int_\tau^{t_0} b(s) ds +\psi(x_0,t_0) -\psi(\gamma(\tau),\tau)-\int_\tau ^t L(s,\gamma(s),\alpha(s)) \: ds \leq 0  \: .
\ee

We need to distinguish two cases depending if $x_0$ is or not on the junction.  

{\bf Case 1:  If $ x_0 \in \Omega_i$ for $i=1,2$ (away from the junction). } 

Let $G_i(t,x,p) \in C([0,T] \times \R \times \R)$   and $G_i$ satisfies for some $\delta >0 $, 
\be \label{stimaGOmega1}
-b(t)+G_i(t,x,p) \leq H(t,x,p) \mbox{ for }  |p -  \psi_x(x_0,t_0) | \leq \delta \: , |x -x_0| \leq \delta  \mbox{ a.e. }  |t - t_0| \leq \delta  .
 \ee
 (We suppose here $\delta$ small enough to have $|x -x_0| \leq \delta \Rightarrow x \in  \Omega_i$.)   \\
 Our aim is to prove that 

  \[
   \psi_t(x_0, t_0)+ G_i(t_0, x_0, \psi_x(x_0,t_0) ) \leq 0  . 
 \]
 
We consider now any admissible trajectory with constant control $(\gamma, \bar{\alpha_i}) \in  {\T}_{\tau,\gamma(\tau)}^{t_0,x_0}$  and $\tau-t_0$ small enough to have for any $s\in  [\tau , t_0]$,  $\gamma(s) \in \Omega_i$  and  $\vert \gamma(s)- x_0 \vert \leq \min\{\delta, \eta\}$ so that \eqref{subALLOmega1} and   \eqref{stimaGOmega1}
 apply.  Thus, arguing as in the classical case,    
 $$
\int_\tau^{t_0} b(s) ds +\int_\tau^{t_0}  \left[ \psi_t(\gamma(s),s) + \psi_x(\gamma(s),s) \bi(s, \gamma(s), \bar{\alpha_i} )  -  \li(s, \gamma(s), \bar{\alpha_i}) \right] \: ds  \leq 0.
$$

Taking the supremum over all $\bar{\alpha_i} \in \A_i$ we have 
$$
\int_\tau^{t_0} b(s) ds +\int_\tau^{t_0}  \left[   \psi_t(\gamma(s),s) +   H_i(s,\gamma(s), \psi_x(\gamma(s),s))  \right]  \: ds  \leq 0 
$$ 
that by \eqref{stimaGOmega1} gives 
$$
 \int_\tau^{t_0}   \psi_t(\gamma(s),s) +   G_i(s,\gamma(s), \psi_x(\gamma(s),s)) \: ds  \leq 0. 
$$ 
Conclusion 
then easily follows by dividing by $t_0-\tau$, letting $\tau \ds t_0$ and recalling that $G_i$ is continuous.

{\bf Case 2: If $ x_0 =0$ (on the junction).} 

 Let $G_0 \in C([0,T] \times \overline{\Omega_1} \times \overline{\Omega_2}  \times \R^2)$ be such that, for some $\delta >0 $ 
 \be \label{Gsubdef2}
-b(s)+G_0(s,x_1, x_2 ,(p_1,p_2)) \leq \max\{ A(s), \Hu^-(s,x_1,p_1), \Hd^+(s,x_2,p_2) \}    
\ee
for $|p_1 - \psi_x(0^-,t_0) | \leq \delta $,  $|p_2 - \psi_x(0^+,t_0) | \leq  \delta$, $|x_1|, |x_2| \leq \delta$ and a.e.  $| s- t_0| \leq \delta$, our aim is to prove that 

$$
 \psi_t(0,t_0)+ G_0(t_0, 0,0, (\psi^1_x(0^+,t_0),\psi^2_x(0^-,t_0)) ) \leq 0. 
 $$

We consider now in \eqref{subALLOmega1} three different choices of $(\gamma, \alpha) \in  {\T}_{\tau,\gamma(\tau)}^{t_0,0}$ as follows. \\
{\bf (A)} For any constant control  $\bar{\alpha}_1 \in A_1$  such that $b_1(t_0,0,\alpha_1)  >0 $, by convexity and normal controllability assumptions \eqref{ipRC2}-\eqref{IpControl22} there exists an admissible trajectory 
$(\gamma_1, \bar{\alpha}_1) \in  {\T}_{\tau,\gamma(\tau)}^{t_0,0}$ that remains in $\Omega_1$ for $\tau-t_0$ small enough. 
Therefore,   \eqref{stimatestA2}-\eqref{subALLOmega1} reads in this case
 $$
\int_\tau^{t_0} b(s) ds +\int_\tau^{t_0} (\psi_1)_x(0,t)b_1(s,\gamma_1(s),\bar{\alpha}_1)  - l_1( s,\gamma_1(s),\bar{\alpha}_1)   \: ds+ \psi_t(0,t_0)(t_0-\tau)+ o(t_0-\tau) \leq 0.
$$
Thus 
$$
\frac{1}{(t_0-\tau)} \int_\tau^{t_0}  \Big[ b(s) + (\psi_1)_x(0,t_0) b_1(s,\gamma_1(s),\bar{\alpha}_1)  - l_1( s,\gamma_1(s),\bar{\alpha}_1) \Big]  ds  + \psi_t(0,t_0) \leq o(1) \: .
$$
Fix $s \in  (\tau,t_0)$,  taking the supremum over all $\alpha_1$ such that $b_1(s,\gamma_1(s),\alpha_1)< 0$ and recalling Lemma \ref{lempropH}  we obtain that 
\be \label{unoA2}
\frac{1}{(t_0-\tau)}\int_\tau^{t_0} \left[ b(s)  + H_1^-(s,\gamma_1(s),(\psi_1)_x(0,t))  \right]\: ds + \psi_t(0,t_0) \leq o(1).
\ee
{\bf (B)} For any constant control  $\bar{\alpha}_2 \in \A_2$ such that $b_2(t_0,0, \bar{\alpha}_2)  >0 $ we can argue like case (A) and similarly obtain
\be \label{dueA2}
\frac{1}{(t_0-\tau)}\int_\tau^{t_0} \left[ b(s) + H_2^+(s,\gamma_2(s),(\psi_2)_x(0,t_0)) \right] \: ds + \psi_t(0,t_0) \leq o(1).
\ee
{\bf (C)}  Since $0 \in \A_0$ by assumption \eqref{ipdinamicain0},  the trajectory $\dot{\gamma}(s)=0, \forall s \in (\tau,t_0)$, $\gamma(t_0)=0$ is admissible. Therefore \eqref{stimatestA2}-\eqref{subALLOmega1} become
\be \label{zeroA2}
\frac{1}{(t_0-\tau)}\int_\tau^{t_0}  \left[ b(s)  + A(s)  \right] \: ds+ \psi_t(0,t_0)  \leq o(1)
\ee 
Putting together \eqref{unoA2}, \eqref{dueA2}, \eqref{zeroA2}, we have 
$$
\frac{1}{(t_0-\tau)}\int_\tau^{t_0} \Big[ b(s) +\max\{ A(s), H_1^-(s,\gamma_1(s),(\psi_1)_x(0,t_0)),H_2^+(s,\gamma_2(s), (\psi_2)_x(0,t_0)) \}  \Big]\: ds+ \psi_t(0,t_0)  \leq o(1).
$$
By \eqref{Gsubdef2}, for   $s \in (\tau,t_0)$, $t_0 - \tau \leq \delta$,  
$$
\begin{array}{l}
b(s) \geq   G_0(s, \gamma_1(s), \gamma_2(s), (\psi_1)_x(0,t_0), (\psi_2)_x(0,t_0))+  \\
\quad \quad  \quad \quad \quad \quad  -\max\{ A(s), H_1^-(s,\gamma_1(s),(\psi_1)_x(0,t_0)),H_2^+(s,\gamma_2(s),(\psi_2)_x(0,t_0)) \} 
\end{array}
$$
therefore, for $t_0 - \tau \leq \min\{\delta,\eta\} $, 
$$
\frac{1}{(t_0-\tau)}\int_\tau^{t_0}  G_0(s,  \gamma_1(s),\gamma_2(s), ((\psi_1)_x(0,t_0), (\psi_2)_x(0,t_0)))  \: ds+ \psi_t(0,t_0)  \leq o(1)
$$
Thanks of the continuity of $G_0$ the conclusion follows then by letting $\tau$ tends to $t_0$.

{\bf Step 2 (Supersolution). } 

Let $\psi \in PC^1(\R \times (0,T))$, $b \in L^1(0,T)$  and $(x_0, t_0)$ be a local minimum point in  $\R \times (0,T)$ of 
$ \dis u(x,t) -(  \int_0^t b(s) ds + \psi(x,t))$. 

By DPP,  for any $\ve >0$, there exists  a trajectory $(\gamma_\ve, \alpha_\ve) \in  {\T}_{\tau,\gamma_\ve(\tau)}^{t_0,x_0}$ such that 
 $$
  u(x_0,t_0) +\ve (t_0- \tau) >  u(\gamma_\ve(\tau),\tau) + \int_\tau ^{t_0} L(s,\gamma_\ve(s),\alpha_\ve(s)) \: ds .
 $$
Thus, there exists  $\eta >0$ such that  for $\tau < t$, $\vert \tau - t_0 \vert \leq \eta \:, \: \vert \gamma_\ve(\tau) - x_0\vert \leq \eta$, we have 
\be \label{supALLOmega1}
 \int_\tau^{t_0} b(s) ds+\psi(x_0,t_0)- \psi(\gamma_\ve(\tau),\tau)-  \int_\tau ^{t_0} L(s,\gamma_\ve(s),\alpha_\ve(s)) \: ds \geq -\ve(t_0-\tau).
 \ee

We need to distinguish two cases depending if $x_0$ is or not on the junction.  

{\bf Case 1:  If $ x_0 \in \Omega_i$ for $i=1,2$ (away from the junction). } 

Let $G_i(t,x,p) \in C([0,T] \times \R \times \R)$   and $G_i$ satisfies for some $\delta >0 $, 
 \be \label{stimaGOmega1super}
-b(t)+G_i(t,x,p) \geq H_i(t,x,p) \mbox{ for }  |p - \psi_x(x_0,t_0) | \leq \delta \: , |x -x_0| \leq \delta  \mbox{ and a.e.}  |t - t_0| \leq \delta  .
 \ee
(We suppose here $\delta$ small enough to have $|x -x_0| \leq \delta \Rightarrow x \in  \Omega_i$.)  \\
Our aim is to prove that 
$$
   \psi_t(x_0, t_0)+ G_i(t_0, x_0, \psi_x(x_0,t_0) ) \geq 0  .
$$
Take now $\tau-t_0$ small enough to have  $\gamma_\ve(s) \in \Omega_i, \forall s\in  [\tau , t_0]$ and  $\vert \gamma_\ve(s)- x_0 \vert \leq \min\{\delta, \eta\}$ so that \eqref{supALLOmega1} and  \eqref{stimaGOmega1super} apply.
Thus, 
$$
\int_\tau^{t_0} b(s) ds +\int_\tau^{t_0}  \Big[ \psi_t(\gamma_\ve(s),s) + \psi_x(\gamma_\ve(s),s) \bi(s, \gamma_\ve(s), \alpha_\ve(s) )  -  \li(s, \gamma_\ve(s), \alpha_\ve(s) ) \Big] \: ds   \geq -\ve(t_0-\tau).
$$
Observe now that for any $s \in   [\tau , t_0]$, 
$$
\sup_{\alpha_i \in \A_i}  \: \{   \psi_x(\gamma_\ve(s),s) \bi(s, \gamma_\ve(s), \alpha_1 ) - \li(s, \gamma_\ve(s), \alpha_1 ) \} \geq \psi_x(\gamma_\ve(s),s) \bi(s, \gamma_\ve(s), \alpha_\ve(s) ) - \li(s, \gamma_\ve(s), \alpha_\ve(s) ) 
$$ 
therefore  
$$
 \int_\tau^{t_0} b(s)   + \psi_t(\gamma_\ve(s),s) ds  +\int_\tau^{t_0} H_i( s, \gamma_\ve(s),\psi_x(\gamma_\ve(s),s) ) ds  \geq -\ve(t_0-\tau)
$$
and by \eqref{stimaGOmega1super}, 
$$
 \int_\tau^{t_0} \psi_t(\gamma_\ve(s),s) ds  + G_i( s, \gamma_\ve(s),\psi_x(\gamma_\ve(s),s) ) ds  \geq -\ve(t_0-\tau).
$$
Conclusion then easily follows by dividing by $t_0-\tau$, $\tau \ds t_0$ and recalling that $G_i$ is continuous.

{\bf Case 2: If $ x_0 =0$ (on the junction).} 

Let $G_0 \in C([0,T] \times \overline{\Omega_1} \times \overline{\Omega_2}  \times \R^2)$ be such that, for some $\delta >0 $ 
 \be \label{Gsubdef2A}
-b(s)+G_0(s,x_1, x_2 ,(p_1,p_2)) \geq \max\{ A(s), \Hu^-(s,x_1,p_1), \Hd^+(s,x_2,p_2) \}    
\ee
for $|p_1 - \psi_x(0^-,t_0) | \leq \delta $,  $|p_2 - \psi_x(0^+,t_0) | \leq  \delta$, $|x_1|,|x_2| \leq \delta$ and a.e.  $| s- t_0| \leq \delta$, our aim is to prove that 
$$
 \psi_t(0,t_0)+ G_0(t_0, 0,0 , (\psi^1_x(0^+,t_0),\psi^2_x(0^-,t_0)) ) \geq 0. 
$$
In the particular case of $(\gamma_\ve, \alpha_\ve) \in  {\T}_{\tau,\gamma_\ve(\tau)}^{t_0,0}$  only staying in $\Omega_1$, $\Omega_2$ or $\{0\}$ the conclusion easily follows arguing as above in Case 1. \\ 
Let now $I=(\Iu, \Id, \I0)$ be  the partition of  
 $(0,t)$  of measurable sets such that $ \gamma_\ve(s) \in \Omega_i$ if $s \in I_i$, $i=1,2$, $\gamma_\ve(s)=0$ if $s \in I_0$. 
 Suppose  that $I_i \not = \emptyset$ for at least two of the three sets $I_i$. Assume, for example that is $I_1 \not = \emptyset$,  we write 
$$
 \{ s \in (\tau,t_0) \: :  \: \gamma_\ve(s)  >0  \}= \bigcup_{k}  \: \left] s^\ve_k, s^\ve_{k+1} \right [  .
$$
Note that $ \gamma_\ve(s^\ve_{k+1})=0$ and  $ \gamma_\ve(s^\ve_{k}) > 0$  if $s^\ve_{k} > \tau$ but if $s^\ve_{k}=\tau$ for some $k$, one can have $ \gamma_\ve(s^\ve_{k}) =0$.
Thus, 
$$
\frac{1}{s^\ve_{k+1}-s^\ve_k} \int_{s^\ve_k}^{s^\ve_{k+1}} b_1(s, \gamma_\ve(s), \alpha^\ve_1(s)) ds = \frac{\gamma_\ve(s^\ve_{k+1})-\gamma_\ve(s^\ve_{k}) }{s^\ve_{k+1}-s^\ve_k} \leq 0 
$$
therefore,  by convexity and normal controllability assumptions  \eqref{ipRC2}-\eqref{IpControl22} and a measurable selection theorem (e.g. Corollary of Proposition 1 in \cite{Roc}) there exists  a control $\bar{\alpha_1}$  
such that  $b_1(s, \gamma_\ve(s), \bar{\alpha_1}(s)) \leq  0$
for $s \in [s^\ve_k,s^\ve_{k+1}]$ and 
 $$
 \int_{s^\ve_k}^{s^\ve_{k+1}} \Big( b_1(s, \gamma_\ve(s), \alpha^\ve_1(s)), l_1( s, \gamma_\ve(s),\alpha^\ve_1(s)) \Big) ds =   \int_{s^\ve_k}^{s^\ve_{k+1}} \Big( b_1(s, \gamma_\ve(s), \bar{\alpha_1}(s) ), l_1( s, \gamma_\ve(s),\bar{\alpha_1}(s)) \Big) ds +o(s^\ve_{k+1}-s^\ve_k). 
 $$
Thus
$$
\begin{array}{l}
\dis \int_{s^\ve_k}^{s^\ve_{k+1}} (\psi_1)_t(\gamma_\ve(s),s) + (\psi_1)_x(\gamma_\ve(s),s) b_1(s, \gamma_\ve(s), \alpha^\ve_1(s)) - l_1(s, \gamma_\ve(s),\alpha^\ve_1(s))   \: ds     \\
\dis =\int_{s^\ve_k}^{s^\ve_{k+1}} \psi_t(0,t) + (\psi_1)_x(0,t) b_1(s, \gamma_\ve(s), \alpha^\ve_1(s))   - l_1(s, \gamma_\ve(s),\alpha^\ve_1(s))   \: ds + o(s^\ve_{k+1}-s^\ve_k)  \\
\dis = \int_{s^\ve_k}^{s^\ve_{k+1}}  \big[ \psi_t(0,t) + (\psi_1)_x(0,t) b_1(s, \gamma_\ve(s), \bar{\alpha_1} (s))- l_1( s, \gamma_\ve(s),\bar{\alpha_1}(s)) \big] \: ds + o(s^\ve_{k+1}-s^\ve_k)  \\
\dis \leq \int_{s^\ve_k}^{s^\ve_{k+1}}   \Hu^-(s, \gamma_\ve(s),(\psi_1)_x(0,t)) \: ds +  \psi_t(0,t_0)(s^\ve_{k+1}-s^\ve_k) +  o(s^\ve_{k+1}-s^\ve_k) \: .
\end{array}
$$
Therefore, using that by construction $ \dis \int_{I_1}=   \sum_{k} \int_{s^\ve_k}^{s^\ve_{k+1}} $ we have 
$$
\begin{array}{l}
\dis  \int_{I_1} (\psi_1)_t(\gamma_\ve(s),s) + (\psi_1)_x(\gamma_\ve(s),s)   b_1(s, \gamma_\ve(s), \alpha^\ve_1(s)) - l_1(s, \gamma_\ve(s),\alpha^\ve_1(s))   \: ds    \\

\dis \leq   \int_{I_1}  \Hu^-(s, \gamma_\ve(s),(\psi_1)_x(0,t)) \: ds + \psi_t(0,t_0)   \vert I_1 \vert + o(t_0-\tau). \\
 \end{array} 
 $$

In the case of $I_2 \not = \emptyset$,  we can write 
$$
 \{ s \in (\tau,t_0) \: :  \: \gamma_\ve(s)  < 0  \}= \bigcup_{k}  \: \left] s^\ve_k, s^\ve_{k+1} \right [
$$
and arguing similarly obtain

$$
\begin{array}{l}
\dis  \int_{I_2} (\psi_2)_t(\gamma_\ve(s),s) + (\psi_2)_x(\gamma_\ve(s),s)  b_2(s, \gamma_\ve(s),\alpha^\ve_2(s))  - l_2(s, \gamma_\ve(s), \alpha^\ve_2(s))   \: ds    \\

\dis \leq \int_{I_2}   \Hd^+(s,\gamma_\ve(s),(\psi_2)_x(0,t_0)) \: ds + \psi_t(0,t_0)  \vert I_2 \vert + o(t_0-\tau) \\
\end{array} 
$$

Note that,  since $t_0 \in I_0$
$$
\begin{array}{r}
\dis  \int_{I_0} (\psi_t(0,s)   + A(s)  ) \: ds   \dis  \leq    \psi_t(0,t_0) \vert I_0 \vert   + \int_{I_0}  A(s)  ds + o(t_0-\tau) \: .\\
 \end{array} 
 $$
Since $  I_0 \cup  I_1 \cup  I_2  =(t-\tau)$, we can use the above estimates in \eqref{stimatestA2}-\eqref{supALLOmega1} and conclude  
that 
$$
\begin{array}{rl}
\dis  -\ve(t_0-\tau) & \dis  \leq  \int_\tau^{t_0} b(s) ds+\psi(0,t_0)- \psi(\gamma_\ve(\tau),\tau)-  \int_\tau ^{t_0} L(s,\gamma_\ve(s),\alpha_\ve(s)) \: ds  \\
\dis  & \dis \leq  \int_\tau^{t_0} b(s) ds +  \psi_t(0,t_0) ( t_0-\tau) +  \int_{I_2}   \Hd^+(s,\gamma_\ve(s),(\psi_2)_x(0,t_0)) \: ds+ \\
& \dis \qquad \qquad +\int_{I_1}  \Hu^-(s, \gamma_\ve(s),(\psi_1)_x(0,t_0)) \: ds +  \int_{I_0}  A(s)  ds + o(t_0-\tau) \leq \\
& \dis  \leq \int_\tau^{t_0} G_0(s, \gamma_\ve(s),\gamma_\ve(s),  (\psi_1)_x(0,t_0), (\psi_2)_x(0,t_0)) ds  +  \psi_t(0,t_0) ( t_0-\tau)+ o(t_0-\tau) 
 \end{array} 
$$
where we also used  \eqref{Gsubdef2A}.  Dividing by $(t_0-\tau)$ and letting $\tau \ds t_0$ we finally have that 
$$
 \psi_t(0,t_0) + G_0(t_0, 0,0, (\psi_1)_x(0,t_0), (\psi_2)_x(0,t_0)) \geq 0  \: .
$$

 \hfill $\Box$

\section{On possible generalizations} \label{gene}

In this section we aim to discuss some generalizations of Definition \ref{defHmismod2}  and comparison results Theorem \ref{compmod2LOC} and Remark \ref{gendim1}, in particular weakening the convexity assumption.  

Beyond  the construction an approximating sequence, the key ingredient in the proof of the comparison principle  is the holding of a comparison result for the approximating equations  \eqref{eqHJmod2}$_n$.   Roughly speaking,  one can mimic the proof of Theorem \ref{compmod2LOC}   in any framework for which a comparison principle holds for the approximating  continuous Hamiltonians and flux limiter. 
Let us now discuss how to exploit this idea in some natural generalisation.

{\em Non-convex Hamiltonians.} 

The theory of Flux-limited solution has been developed in \cite{ImMo1}  directly in the case of quasi-convex and coercive Hamiltonians. 
Therefore \cite[Theorem 5.8]{ImMo1} and consequently Theorem \ref{compmod2LOC},  still hold if one replace assumption \eqref{ipHami2Conv}  with 
\be \label{ipHami2Quasiconv}
H_i(t,x, \cdot) : \R \ds \R \mbox{ are  quasi-convex  and coercive uniformly with respect to } (t,x).
\ee
(For the quasi-convex case, one can also refer to \cite[Section 14.3.2]{BaChas}. ) \\
Moreover, in the case of Hamiltonian depending only on the gradient variable, i.e. equation \eqref{eqHJmodIntro}, N. Forcadel and R. Monneau proved in  \cite{ForMon} a comparison result assuming Hamiltonians to be only continuous and coercive and continuous Flux limiter. Based on \cite[Thorem 6.1]{ForMon} our result can then be extended to equation \eqref{eqHJmodIntro} with $A \in L^\infty(0,T)$ and $H_i(\cdot)$ continuous and coercive.   
 
\newpage
{\em Hamiltonian depending also on the $u$-variable.}  

Result can be easily extended to Hamiltonians continuously depending also on $u$ by classically adding to regularity assumptions on $H$ a $u$-monotonicity hypothesis. Precisely, in the $t$ continuous case, comparison results  \cite[Theorem 8.1]{Is85} and \cite[Theorem 14.2 or 14.3]{BaChas}  hold if we add to assumptions of Theorem  \ref{compmod2LOC}  the following
$$
 \exists \nu  \in \R  \mbox{ s. t. }   u  \ds H(t,x,u,p)- \nu u \mbox{  is nondecreasing in } \R \:, \forall (x,p)\in \R \times \R, \mbox{ a.e. } t \in (0,T).
$$
Therefore, also adding the $u$ dependence in assumption (AP)-\eqref{ipHam2inTmodulus},  we can adapt the proof of Theorem \ref{compmod2LOC} by taking the sup also on $|u |\leq K$ in the definition of $k_n$, i.e. in  \eqref{defknLOC}-\eqref{defknoudLOC} in Step 2. (Note that this is not restrictive  when dealing with bounded solutions.)

{\em Networks}

As in the original paper of C. Imbert and R. Monneau \cite{ImMo1},  the definition of Flux Limited solutions  given for one junction on the line can be naturally extended to networks,  i.e. where more junctions are allowed.  Indeed,  in \cite[Section 5]{ImMo1} Hamilton Jacobi equation are defined on networks and a comparison principle is proved (Theorem 5.8) for continuous, quasi convex and coercive Hamiltonians. We stress here that assumptions (H0)-(H6), (A0)-(A2) in Section 5.3 are  the analogous of hypotheses \eqref{ipHam2inX}, \eqref{ipHam2inP}, \eqref{ipHami2Quasiconv}, \eqref{ipA02}, \eqref{ipA2}. The only difference is that, for networks,  all bounds and constants are assumed to be uniform with respect to the set of edges and vertices.  It is then  our belief that Definition \ref{defHmismod2} and the results in Section  \ref{onedimconvex} can be extended to networks by only requiring the same uniformity assumptions as in the continuous case in all hypotheses.

\vspace{0.3cm} 

{\bf Acknowledgments.} 
The author would like to thank Pierre Cardaliaguet for introducing her to this problem and Regis Monneau for the fruitful discussions. 


\end{document}